\documentclass[reqno]{amsart}
\usepackage{fancyhdr}
\usepackage{txfonts}
\usepackage{amsmath,amsthm,amssymb,hyperref}
\usepackage{tikz}
\usepackage[all]{xy}
\usetikzlibrary{decorations.markings}
\usetikzlibrary{backgrounds,shapes}
\usepackage{subfig}

\newcommand\RR{{\mathcal R}}
\newcommand\RX{{\mathcal X}}
\newcommand\Z{{\mathbb Z}}
\newcommand\C{{\mathbb C}}
\newcommand\Q{{\mathbb Q}}
\DeclareMathOperator\Aut{Aut}
\DeclareMathOperator\Out{Out}
\DeclareMathOperator\Mod{Mod}
\DeclareMathOperator\Inn{Inn}
\DeclareMathOperator\Hom{Hom}
\DeclareMathOperator\PGL{PGL}
\DeclareMathOperator\GL{GL}

\begin{document}

\newtheorem{theorem}{Theorem}[section]
\newtheorem{prop}[theorem]{Proposition}
\newtheorem{lemma}[theorem]{Lemma}
\newtheorem{cor}[theorem]{Corollary}

\title[Representations of surface groups with finite orbits]{Representations of surface 
groups with finite mapping class group orbits}

\author{Indranil Biswas}

\address{School of Mathematics, Tata Institute of Fundamental Research, Homi 
Bhabha Road, Mumbai 400005, India}

\email{indranil@math.tifr.res.in}

\author{Thomas Koberda}

\address{Department of Mathematics, University of Virginia, Charlottesville, VA 
22904-4137, USA}

\email{thomas.koberda@gmail.com}

\author{Mahan Mj}

\address{School of Mathematics, Tata Institute of Fundamental Research, Homi 
Bhabha Road, Mumbai 400005, India}

\email{mahan@math.tifr.res.in}

\author{Ramanujan Santharoubane}

\address{Department of Mathematics, University of Virginia, Charlottesville, VA 
22904-4137, USA}

\email{ramanujan.santharoubane@gmail.com}

\subjclass[2000]{Primary: 57M50; Secondary: 57M05, 20E36, 20F29}

\keywords{Representation variety; surface group; mapping class group; character variety}

\date{\today}

\begin{abstract}
Let $(S,\, \ast)$ be a closed oriented surface with a marked point, let $G$ be a fixed group, and
let $\rho\colon\pi_1(S) \longrightarrow G$ be a representation such that
the orbit of $\rho$ under the action of the mapping class group $\Mod(S,\ast)$ is finite. We prove that the image of $\rho$ is finite. A similar result holds if $\pi_1(S)$ is replaced by the free group $F_n$ on $n\geq 2$
generators and where $\Mod(S,\ast)$ is replaced by $\Aut(F_n)$. We thus resolve a well--known question of M. Kisin. We show that if $G$ is a linear algebraic group and if the representation variety of $\pi_1(S)$ is replaced by the character variety, then there are infinite image representations which are fixed by the whole mapping class group.
\end{abstract}

\maketitle

\section{Introduction}

Let $G$ and $\Gamma$ be groups, and let \[\RR (\Gamma,G)\,:=\,\Hom(\Gamma,G)\] be the 
\emph{representation variety} of $\Gamma$. The automorphism group $\Aut(\Gamma)$ acts on 
$\RR(\Gamma,G)$ by precomposition.

Let $\Gamma\,=\,\pi_1(S)$, where $S$ is a closed, orientable surface of genus at least 
two with a base-point $\ast$. The Dehn--Nielsen--Baer Theorem (see~\cite{FarbMargalit}) 
implies that the mapping class group $\Mod(S,\ast)$ of $S$ which preserves $\ast$ is 
identified with an index two subgroup of $\Aut(\Gamma)$. In this note, we show that if 
$\rho\,\in\,\RR(\Gamma,G)$ has a finite $\Mod(S,\ast)$--orbit, then the image of $\rho$ is 
finite, thus resolving a well--known question which is often attributed to M. Kisin. We 
show that the same conclusion holds if $\Gamma$ is the free group $F_n$ of 
finite rank $n\, \geq\, 2$, and $\Mod(S,\ast)$ is replaced by $\Aut(F_n)$.

\subsection{Main results}

In the sequel, we assume that $S$ is a closed, orientable surface of genus $g\geq 2$ and 
that $F_n$ is a free group of rank at least two, unless otherwise stated explicitly.

\begin{theorem}\label{intromain}
Let $\Gamma\,=\, \pi_1(S) $ or $F_n$, and let $G$ be an arbitrary group. Suppose that
$\rho \,\in\, \RR(\Gamma, G)$ has a finite orbit under the action of $\Aut(\Gamma)$. Then $\rho(\Gamma)$ is finite.
\end{theorem}

Note that if $\Gamma\,=\,\pi_1(S)$ then $\rho$ has a finite orbit under $\Aut(\Gamma)$ if and only if 
it has a finite orbit under $\Mod(S,\ast)$ because $\Mod(S,\ast)$ is a subgroup of
$\Aut(\Gamma)$ of finite index. Note also that if the homomorphism $\rho$ has a finite image
then the orbit of $\rho$ for the action of $\Aut(\Gamma)$ on $\RR(\Gamma, G)$ is finite, because
$\Gamma$ is finitely generated.

We will show by example that Theorem \ref{intromain} fails for a general group $\Gamma$. 
Moreover, Theorem \ref{intromain} fails if $\Gamma$ is a linear algebraic group with 
the representation variety of $\Gamma$ being replaced by the character variety
$\Hom(\Gamma,G)/G$:

\begin{prop}\label{prop:character}
Let $\Gamma\,=\,\pi_1(S)$ and let \[\mathcal{X}(\Gamma,\GL_n(\C))\,:=\,
{\rm Hom}(\Gamma,\GL_n(\C))/\!\!/\GL_n(\C)\] be its $\GL_n(\C)$ character variety. For $n\gg 0$, there
exists a point $\chi\,\in\,\RX(\Gamma,\GL_n(\C))$ such that $\chi$ is the character of a
representation with infinite image while the action of $\Mod(S,\ast)$ on $\RX(\Gamma,\GL_n(\C))$
fixes $\chi$.
\end{prop}

\subsection{Punctured surfaces}

If $S$ is not closed then $\pi_1(S)$ is a free group, and the group $\Mod(S,\ast)$ is 
identified with a subgroup of $\Aut(\pi_1(S))$, though this subgroup does not have finite 
index. The conclusion of Theorem \ref{intromain} remains valid if $S$ is a once--punctured 
surface of genus $g\geq 1$ and $\Aut(\pi_1(S))$ is replaced by $\Mod(S,\ast)$, though it 
fails for general punctured surfaces. See Proposition \ref{prop:braid}.

\section{$\Aut(\Gamma)$--invariant representations}

We first address the question in the special case where the $\Aut(\Gamma)$--orbit of $\rho \,:\, \Gamma
\,\longrightarrow\, G$ on $\RR (\Gamma, 
G)$ consists of a single point.

\begin{lemma}\label{abelian} Let $\Gamma$ be any group, and suppose that $\rho\in\RR (\Gamma, G)$ is $\Aut(\Gamma)$--invariant. Then $\rho(\Gamma)$ is abelian. 
\end{lemma}

\begin{proof} Since $\rho \in \RR (\Gamma, G)$ is invariant under the normal subgroup
$\Inn(\Gamma)\,<\,\Aut(\Gamma)$ consisting of inner automorphisms,
$$\rho(ghg^{-1}) \,=\, \rho (h)$$
for all $g, h \in G$. Hence we have $\rho(g) \rho(h) = \rho(h) \rho(g)$.
\end{proof}

\begin{lemma}\label{trivial}
Let $\Gamma = \pi_1(S)$ or $\Gamma=F_n$, and let
$\rho\colon \Gamma \longrightarrow G$ be $\Aut(\Gamma)$--invariant. Then $\rho(\Gamma)$ is trivial. 
\end{lemma}

\begin{proof}
Without loss of generality, assume that $\rho(\Gamma) \,=\, G$. By Lemma \ref{abelian}, the group 
$\rho (\Gamma)$ is abelian. Hence $\rho$ factors as
\begin{equation}\label{fa}
\rho \,= \,\rho^{ab} \circ A\, ,
\end{equation}
where \[A \,\colon\, 
\Gamma\,\longrightarrow\, \Gamma/[\Gamma,\, \Gamma]\,=\,H_1(\Gamma,\,\Z)\] is the abelianization map, and $\rho^{ab}\colon H_1(\Gamma,\Z) 
\longrightarrow G$ is the induced representation of $H_1(\Gamma,\Z)$. Let $g$ be the genus
of $S$, so that the rank of $H_1(\Gamma,\,\Z)$ is $2g$. Fix a symplectic basis
\[\{a_1,\, \cdots, \, a_g,\, b_1,\, \cdots,\, b_g\}\]
of $H_1(\Gamma,\, \Z)$. So the group of automorphisms of $H_1(\Gamma,\, \Z)$ preserving the cap
product is identified with ${\rm Sp}_{2g}( \Z)$.

Suppose that $\rho(\Gamma)$ is not trivial. Then there exists $z \,\in\, H_1(\Gamma,\,\Z)$ such 
that $\rho^{ab}(z)$ is not trivial. Consider the action of ${\rm Sp}_{2g}( \Z)$ on 
$H_1(\Gamma,\,\Z)$, induced by the action of the mapping class group $\Mod (S, \ast)$ on 
$H_1(\Gamma,\,\Z)$. There is an element of ${\rm Sp}_{2g}( \Z) $ taking $a_i$ to
$a_i + b_i$. Therefore, from the given condition that the action of $\Mod (S, \ast)$ on $\rho$
has a trivial orbit it follows that $\rho^{ab}(a_i)\,=\, \rho^{ab}(a_i + b_i)$, and hence
$\rho^{ab}(b_i) = 0$. Exchanging the roles of $a_i$ and of $ b_i$, we have $\rho^{ab}(a_i) = 0$. Thus $\rho^{ab}$ is a
trivial representation, and hence $\rho$ is also trivial by \eqref{fa}.

A similar argument works if we set $\Gamma\,=\,F_n$. Instead of ${\rm Sp}_{2g}( \Z) $, we 
have an action of $\GL_n(\Z)$ on $H_1(\Gamma,\,\Z)$ after choosing a basis
$\{a_1, \cdots , a_n\}$ for $H_1(\Gamma,\, \Z)$. Then for each $1\,\leq j\leq n$ and $i\neq j$, there exist an element of
$\GL_n(\Z)$ that takes $a_i$ to $a_i + a_j$. This 
implies that $\rho^{ab}(a_j) = 0$ as before.
\end{proof}

There is an immediate generalization of Lemma \ref{trivial} whose proof is identical to 
the one given:

\begin{lemma}
Let $\Gamma$ be group, let \[H_1(\Gamma,\Z)_{\Out(\Gamma)}\,=\,H_1(\Gamma,\,\Z)/\langle \phi(v)-v\mid \textrm{$v
\,\in\, H_1(\Gamma,\,\Z)$ and $\phi\,\in\,\Out(\Gamma)$}\rangle\] be the module of co--invariants of
the $\Out(\Gamma)$ action on $H_1(\Gamma,\,\Z)$, and let $\rho\in\RR(\Gamma,G)$ be an
$\Aut(\Gamma)$--invariant representation of $\Gamma$. If $H_1(\Gamma,\Z)_{\Out(\Gamma)}\,=\,0$ then $\rho(\Gamma)$ is
trivial. If $H_1(\Gamma,\Z)_{\Out(\Gamma)}$ is finite then $\rho(\Gamma)$ is finite as well.
\end{lemma}

\begin{cor}\label{cor:finite index}
Let $\Gamma$ be a closed surface group or a finitely generated free group, and let $H\,<\,
\Aut(\Gamma)$ be a finite index subgroup. Then the module of $H$--co--invariants for $H_1(\Gamma,\,
\Z)$ is finite.
\end{cor}

\begin{proof}
Since $H\,<\,\Aut(\Gamma)$ has finite index, there exists an integer $N$ such that for each
$\phi\,\in\,\Aut(\Gamma)$, we have $\phi^N\,\in\, H$. In particular, the $N^{th}$ powers of the
transvections occurring in the proof of Lemma \ref{trivial} lie in $H$, whence the $N^{th}$
powers of elements of a basis for $H_1(\Gamma,\,\Z)$ must be trivial. Consequently, the module of $H$--co--invariants is finite.
\end{proof}

\section{Representations with a finite orbit}

\subsection{Central extensions of finite groups}

\begin{lemma}\label{finabelian}
Let $\Gamma$ be any group, and let $\rho\colon \Gamma \,\longrightarrow\, G$ be a representation.
Suppose that the orbit of $\rho$ under the action of $\Aut(\Gamma)$ on $\RR (\Gamma, G)$ is
finite. Then $\rho(\Gamma)$ is a central extension of a finite group. 
\end{lemma}

\begin{proof}
The given condition implies that for the action of $\Inn(\Gamma)$ on
$\RR (\Gamma, G)$, the orbit of $\rho$ is finite. Consequently, there exists
a finite index subgroup $\Gamma_1$ of $\Gamma$ that fixes $\rho$ under the inner action. Hence by the
argument in Lemma \ref{abelian}, the group $\rho(\Gamma_1)$ commutes with $\rho(\Gamma)$, so
the center of $\rho(\Gamma)$ contains $\rho(\Gamma_1)$. Since $\Gamma_1$ is of finite
index in $\Gamma$, the result follows.
\end{proof}

\begin{lemma}\label{lem:stabilizer}
Let $\rho$ be as in Lemma \ref{finabelian} and let $H=\mathrm{Stab}(\rho)<\Aut(\Gamma)$ be the stabilizer of $\rho$. Then the center of $\rho(\Gamma)$ is equal to its module of co--invariants under the $H$--action.
\end{lemma}
\begin{proof}
This is immediate, since $\rho(\Gamma)$ is invariant under the action of $H$. Thus, if $z$ lies in the center of $\rho(\Gamma)$ then $\phi(z)=z$ for all $\phi\in H$. In particular, the subgroup of the center of $\rho(\Gamma)$ generated by elements of the form $\phi(z)-z$ is trivial.
\end{proof}

\subsection{Homology of finite index subgroups}

Let $\Gamma$ be a finitely generated group, and let $\rho\,\in\,\RR(\Gamma, G)$ be a representation 
whose orbit under the action of $\Aut(\Gamma)$ is finite. By Lemma \ref{finabelian}, we have 
that $\rho(\Gamma)$ fits into a central extension: \[1\longrightarrow 
Z\longrightarrow \rho(\Gamma)\longrightarrow F\longrightarrow 1\, ,\] where $F$ is a finite group 
and $Z$ is a finitely generated torsion--free abelian group lying in the center of 
$\rho(\Gamma)$.

Consider the 
group $N\,=\,\rho^{-1}(Z)\,<\, \Gamma$. This is a finite index subgroup of $\Gamma$, since $Z$ has 
finite index in $\rho(\Gamma)$. By replacing $N$ by a further finite index subgroup of $\Gamma$ if 
necessary, we may assume that $N$ is characteristic in $\Gamma$ and hence $N$ is invariant under 
automorphisms of $\Gamma$.

Since $Z$ is an abelian group, we have that the restriction of $\rho$ to $N$ factors 
through the abelianization $H_1(N,\,\Z)$. As before, we write $\rho^{ab}\,\colon\, 
H_1(N,\Z)\,\longrightarrow \,Z$ for the corresponding map, and we write $Q=\Gamma/N$. The group 
$\Gamma$ acts by conjugation on $N$ and on $H_1(N,\,\Z)$, and on $Z$ via $\rho$, thus turning both $H_1(N,\,\Z)$ and $Z$ into $\Z[\Gamma]$--modules. Observe that the $\Gamma$--action on $H_1(N,\,\Z)$ turns 
this group into a $\Z[Q]$--module, and that the $\Z[\Gamma]$--module structure on $Z$ is 
trivial. Note that the map $\rho^{ab}$ is a homomorphism of $\Z[\Gamma]$--modules. Summarizing 
the previous discussion, we have that following diagram commutes $\Gamma$--equivariantly:
$$\xymatrix{
 N \ar[d]^{\rho} \ar[r]^{A} & H_1(N,\Z) \ar[dl]^{\rho^{ab}}\\
 Z &
}$$

\subsection{Chevalley--Weil Theory}

Let $\Gamma$ be a group, and let $N\,<\,\Gamma$ be a finite index normal subgroup with quotient group 
\begin{equation}\label{dQ}
Q\,:=\,\Gamma/N\, .
\end{equation}
When $\Gamma$ is a closed surface group or a finitely generated free group, it is 
possible to describe $H_1(N,\,\Q)$ as a $\Q[Q]$--module. We address closed surface groups 
first:

\begin{theorem}[Chevalley--Weil Theory for surface groups,~\cite{cw}, see also ~\cite{gllm, Kob2012}]\label{thm:cw1}
Let $S\,=\,S_g$ be a closed surface of genus $g$, and let $\Gamma\,=\,\pi_1(S)$. Then
there is an isomorphism of $\Q[Q]$--modules (defined in \eqref{dQ})
\[H_1(N,\Q)\,\stackrel{\sim}{\longrightarrow}\, \rho_{reg}^{2g-2}\oplus \rho_0^2\, ,\]
where $\rho_{reg}$ is the regular representation of $Q$ and $\rho_0$ is the trivial representation
of $Q$. Moreover, the invariant subspace of $H_1(N,\,\Q)$ is $\Aut(\Gamma)$--equivariantly isomorphic
to $H_1(\Gamma,\,\Q)$ via the transfer map.
\end{theorem}

The corresponding statement for finitely generated free groups was also observed by Gasch\"utz, and is identical to the statement for surface groups, \emph{mutatis mutandis}:

\begin{theorem}[Chevalley--Weil Theory for free groups,~\cite{cw}, see also ~\cite{gllm, Kob2012}]\label{thm:cw2}
Let $\Gamma\,=\,F_n$ be a free group of rank $n$. Then there is an isomorphism of $\Q[Q]$--modules
\[H_1(N,\Q)\,\stackrel{\sim}{\longrightarrow}\, \rho_{reg}^{n-1}\oplus \rho_0\, ,\] where
$\rho_{reg}$ is the regular representation of $Q$ and $\rho_0$ is the trivial representation of
$Q$. Moreover, the invariant subspace of $H_1(N,\,\Q)$ is $\Aut(\Gamma)$--equivariantly isomorphic to
$H_1(\Gamma,\,\Q)$ via the transfer map.
\end{theorem}

Tensoring with $\Q$, we have a map \[\rho^{ab}\otimes\Q\colon H_1(N,\Q)\longrightarrow 
Z\otimes\Q\] which is a homomorphism of $\Q[\Gamma]$--modules since the natural map $\rho^{ab}\colon H_1(N,\Z)\longrightarrow Z$ is $\Gamma$--equivariant. We decompose $H_1(N,\Q)=V_0\bigoplus 
\big(\oplus_{\chi} V_{\chi}\big)$ according to its structure as a $\Q[\Gamma]$--module, where 
$V_0$ is the invariant subspace and $\chi$ ranges over nontrivial irreducible characters of 
$Q$.

Note that since $N$ is characteristic in $\Gamma$, the group $\Aut(\Gamma)$ acts on 
$H_1(N,\Q)$ and this action preserves $V_0$. Moreover, Theorems \ref{thm:cw1} and 
\ref{thm:cw2} imply that the $\Aut(\Gamma)$--action on $V_0$ is canonically isomorphic to the $\Aut(\Gamma)$ action on 
$H_1(\Gamma,\,\Q)$, by the naturality of the transfer map.

We are now ready to prove the main result of this note:

\begin{proof}[Proof of Theorem \ref{intromain}]
By the discussion above, it suffices to prove that the group $Z$ is finite, or equivalently that 
the vector space $Z\otimes\Q$ is trivial.

Considering the image of each irreducible representation $V_{\chi}$ under 
$\rho^{ab}\otimes\Q$, from Schur's Lemma it is deduced that either $V_{\chi}$ is in the kernel of 
$\rho^{ab}\otimes\Q$ or it is mapped isomorphically onto its image. Since 
$\rho^{ab}\otimes\Q$ is a $\Q[\Gamma]$--module homomorphism and since $Z\otimes\Q$ is a 
trivial $\Q[\Gamma]$--module, we have that $V_{\chi}\,\subset\,\ker\rho^{ab}\otimes\Q$ whenever 
$\chi$ is a nontrivial irreducible character of $Q$. It follows that $Z\otimes\Q$ is a 
quotient of $V_0$.

Since the $\Aut(\Gamma)$--actions on $H_1(\Gamma,\,\Z)$ and on $V_0$ are isomorphic, Corollary 
\ref{cor:finite index} implies that the module of rational $H$--co--invariants for $V_0$ 
is trivial for any finite index subgroup $H\,<\,\Aut(\Gamma)$, meaning \[V_0/\langle \phi(v)-v\,\mid\, 
v\,\in\, V_0\textrm{ and }\phi\,\in\, H\rangle=0.\]

Let $H\,=\,\mathrm{Stab}(\rho)\,<\,\Aut(\Gamma)$ be the stabilizer of $\rho$, which has finite index 
in $\Aut(\Gamma)$ by assumption. Since $\rho$ is $H$--invariant, we have that $Z\otimes\Q$ is 
also $H$--invariant. Let $v\in V_0$ be an element which does not lie in the kernel of 
$\rho^{ab}\otimes\Q$. Since the module of $H$--co--invariants of $V_0$ is trivial, we 
have that \[v_0=\sum_{i=1}^k a_i(\phi(v_i)-v_i)\] for suitable 
vectors $(v_1,\,\cdots,\,v_k)\,\in\, V^k_0$, rational numbers $(a_1,\,\cdots,\,a_k)\,\in\,{\Q}^k$, 
and automorphisms $(\phi_1,\,\cdots,\,\phi_k)\,\in\, H^k$. Applying $\rho^{ab}\otimes\Q$, we 
have \[(\rho^{ab}\otimes\Q)(v_0)=\sum_{i=1}^k a_i\cdot 
(\rho^{ab}\otimes\Q)(\phi(v_i)-v_i).\] Since $\rho$ and $Z$ are both $H$--invariant, we 
have that $(\rho^{ab}\otimes\Q)(\phi(v_i)-v_i)=0$, whence $(\rho^{ab}\otimes\Q)(v_0)=0$. Thus, $v_0\in\ker\rho^{ab}\otimes\Q$, and consequently 
$Z\otimes\Q=0$.
\end{proof}

\section{Counterexamples for general groups}

It is not difficult to see that Theorem \ref{intromain} is false for general groups. We 
have the following easy proposition:

\begin{prop}\label{prop:counterexample}
Let $\Gamma$ be a finitely generated group such that $\Gamma$ surjects to $\Z$ and such that $\Out(\Gamma)$ is
finite. Then there exists a group $G$ and a representation $\rho\,\in\,\RR(\Gamma, G)$ such
that $\rho$ has infinite image and such that the $\Aut(\Gamma)$--orbit of $\rho$ is finite.
\end{prop}

\begin{proof}
Set $$G\,=\,\Gamma^{ab}\, ,$$ and let $\rho\, :\, \Gamma\,\longrightarrow\,
G$ be the abelianization map. Since $\Out(\Gamma)$ is finite, we have
that $\Aut(\Gamma)$ induces only finitely many distinct automorphisms of $G$, and hence $\rho$ has
a finite orbit under the $\Aut(\Gamma)$ action on $\rho\in\RR(\Gamma,G)$.
\end{proof}

It is easy to see that Proposition \ref{prop:counterexample} generalizes to the case 
where $\rho$ has infinite abelian image with $G$ being an arbitrary group. 

There are many natural classes of groups which satisfy the hypotheses of Proposition 
\ref{prop:counterexample}. For instance, one can take a cusped finite volume hyperbolic 
$3$--manifold or a closed hyperbolic $3$--manifold with positive first Betti number; 
every closed hyperbolic $3$--manifold has such a finite cover by the work of 
Agol~\cite{A}. The fundamental groups of these manifolds are finitely generated with 
infinite abelianization, and by Mostow Rigidity, their groups of outer automorphisms are 
finite.

Another natural class of groups satisfying the hypotheses of Proposition 
\ref{prop:counterexample} is the class of random right-angled Artin groups, in the sense 
of Charney--Farber~\cite{CF}. Every right-angled Artin group has infinite abelianization, 
though many have infinite groups of outer automorphisms. Certain graph theoretic 
conditions which are satisfied by finite graphs in a suitable random model guarantee that 
the outer automorphism group is finite, however. An explicit right-angled Artin group 
with a finite group of outer automorphisms is the right-angled Artin group on the 
pentagon graph.

Let $D_n$ denote the disk with $n$ punctures. The mapping class group $\Mod(D_n,\partial 
D_n)$ is identified with the braid group $B_n$ on $n$ strands, and naturally sits inside 
of $\Aut(F_n)\,=\,\Aut(\pi_1(D_n))$. The following easy proposition illustrates another 
failure of Theorem \ref{intromain} to generalize:

\begin{prop}\label{prop:braid}
Let $G$ be a group which contains an element of infinite order. Then there exists an infinite
image representation $\rho\,\in\,\RR(F_n,G)$ which is fixed by the action of $B_n\,<\,\Aut(F_n)$.
\end{prop}
\begin{proof}
Small loops about the punctures of $D_n$ can be connected to a base-point on the boundary of $D_n$ in order to obtain a free basis for $\pi_1(D_n)$. Since the braid group consists of isotopy classes of homeomorphisms of $D_n$, we have that $B_n$ acts on the homology classes of these loops by permuting them. Therefore, we may let $\rho$ be the homomorphism $F_n\longrightarrow\Z$
obtained by taking the exponent sum of a word in the chosen free basis for $\pi_1(D_n)$, and then
sending a generator for $\Z$ to an infinite order element of $G$. It is clear from this
construction that $\rho$ is $B_n$--invariant and has infinite image.
\end{proof}

\section{Character varieties}

In this section we prove Proposition \ref{prop:character}, which relies on one of the 
results in~\cite{KS}. Recall that if $S$ is an orientable surface with negative Euler 
characteristic then the Birman Exact Sequence furnishes a normal copy of $\pi_1(S)$ 
inside of the pointed mapping class group $\Mod(S,\ast)$ 
(see~\cite{Birman,FarbMargalit}).

\begin{theorem}[cf. ~\cite{KS}, Corollary 4.3]\label{thm:KS}
There exists a linear representation $\rho\,\colon\,\Mod(S,\ast)\,\longrightarrow\,\PGL_n(\C)$ such
that the restriction of $\rho$ to $\pi_1(S)$ has infinite image.
\end{theorem}

We remark that in Theorem \ref{thm:KS}, it can be arranged for the image of $\pi_1(S)$ 
under $\rho$ to have a free group in its image, as discussion in~\cite{KS}. Theorem \ref{thm:KS} implies Proposition 
\ref{prop:character} without much difficulty.

\begin{proof}[Proof of Proposition \ref{prop:character}]
Let a representation $\sigma\,\colon\,\Mod(S,\ast)\,\longrightarrow\,\PGL_n(\C)$ be given as
in Theorem \ref{thm:KS}. Choose an arbitrary embedding of $\PGL_n(\C)$ into $\GL_m(\C)$ for
some $m\,\geq \,n$, and let $\rho$ be the corresponding representation of $\Mod(S,\ast)$ obtained
by composing $\sigma$ with the embedding. We will write $\chi$ for its character, and we
claim that this $\chi$ satisfies the conclusions of the proposition.

That $\chi$ corresponds to a representation of $\pi_1(S)$ with infinite image is 
immediate from the construction. Note that $\chi$ is actually the character of a 
representation of $\Mod(S,\ast)$, and that $\Inn(\Mod(S,\ast))$ acts on 
$\RX(\Mod(S,\ast))$ trivially. It follows that $\Inn(\Mod(S,\ast))$ fixes $\chi$ even 
when $\chi$ is viewed as a character of $\pi_1(S)$, since $$\pi_1(S)\,<\,\Mod(S,\ast)$$ is 
normal. The conjugation action of $\Mod(S,\ast)$ on $\pi_1(S)$ is by automorphisms via 
the natural embedding $$\Mod(S,\ast)<\Aut(\pi_1(S))\, .$$ It follows that $\chi$ is invariant 
under the action of $\Mod(S,\ast)$, the desired conclusion. \end{proof}
		
\section*{Acknowledgements}

The authors thank B. Farb for many comments which improved the paper. IB and MM acknowledge support of their respective J. C. Bose Fellowships. TK is partially 
supported by Simons Foundation Collaboration Grant number 429836.

\end{document}